\documentclass[leqno,twoside, 11pt]{amsart}
\usepackage{amssymb, latexsym, esint}

\theoremstyle{plain}
\numberwithin{equation}{section} \numberwithin{figure}{section}

\newtheorem{theorem}{Theorem}[section]
\newtheorem{lemma}[theorem]{Lemma}

\newtheorem{corollary}[theorem]{Corollary}
\newtheorem{definition}[theorem]{Definition}

\theoremstyle{definition}
\newtheorem{remark}[theorem]{Remark}
\newtheorem{example}[theorem]{Example}

\begin{document}
\title[discrete energies on prestrained bodies] {On the variational
  limits of lattice energies on prestrained elastic bodies}
\author{Marta Lewicka and Pablo Ochoa}
\address{Marta Lewicka, University of Pittsburgh, Department of Mathematics, 
139 University Place, Pittsburgh, PA 15260}
\address{Pablo Ochoa, University of Pittsburgh, Department of Mathematics, 
139 University Place, Pittsburgh, PA 15260; and 
I.C.B. Universidad Nacional de Cuyo} 
\email{lewicka@pitt.edu, pdo2@pitt.edu}

\date{\today}
\begin{abstract} 
We study the asymptotic behaviour of the discrete elastic energies in
presence of the prestrain metric $G$, assigned on the continuum reference
configuration $\Omega$.
When the mesh size of the discrete lattice in $\Omega$ goes to zero,
we obtain the variational bounds on 
the limiting (in the sense of $\Gamma$-limit) energy. In case of
the nearest-neighbour and next-to-nearest-neighbour interactions, we derive a precise asymptotic formula,
and compare it with the non-Euclidean model energy relative to $G$.
\end{abstract}

\maketitle

\section{Introduction}

Recently, there has been a growing interest in the study of
prestrained materials, i.e. materials which assume non-trivial rest configurations  
in the absence of exterior forces or boundary conditions. This
phenomenon has been observed in contexts such as: naturally
growing tissues, torn plastic sheets, specifically engineered polymer
gels, and many others.
The basic mathematical model, called ``incompatible elasticity'' has been put forward
in \cite{RHM, Experiment1, Experiment} and further studied in \cite{LP,
  a1, a2, a4, lm1, lm2, k0, k1, k2, k3}.
In this paper, we pose the following question: is it possible to
derive an equivalent continuum mechanics model starting from an
appropriate discrete description, by means of a homogenization procedure when
the mesh size goes to $0$?  Discrete-to-continuum limits of this
type have been investigated by means of $\Gamma$-convergence in a
number of areas of application, including nonlinear elasticity
\cite{AC, Raoult, r, r1, s2, s1} and others (see for example \cite{r2,
  r3, r4, ss}).

Discrete lattices may model both the atomic structures
and mechanical trusses. The latter case is not restricted to classical
material mechanics but it also encompasses biological tissues. 
For instance, in the cell-to-muscle homogenization problem \cite{cmr,
  cmr2, mphd}, the muscle tissue of the heart, which forms a thick
middle (myocardial) layer between the
outer epicardium and the inner endocardium layers,  
is regarded as a set of basic nodes and fibers suitably
arranged.  The myocardial fibers consist of myocytes; these are 
elongated structures which can undergo further
elongation/traction as well as angle interaction.
It is possile to reconstruct an elastic law for the myocardium from the 
known behavior of the myocytes \cite{cmr, cmr2},  and the obtained results are 
consistent with the experimental measurements in the physiological
literature.
Further observations \cite{pcom} confirm that there should be a spatial
heterogeneity in the myocardium cells, as a consequence of the
temporal heterogeneity. Nevertheless, so far measurements were not precise 
enough (due to the noise in the diffraction techniques) to give distinct values,
and therefore most of the time heterogeneity has been left aside in prior works. 

The  analysis in the present paper investigates the relation of the continuum limit of the atomistic models
taking into account the weighted pairwise interactions of nodes in the
lattice, with the continuum elastic energy 
where all possible interactions are taken into account. We show that, although
the limit model inherits the same structure of the 
continuum energy, the two models differ by: (i) the relaxation in the density potential
which, as one naturally expects, is the quasiconvexification of
the original density, and (ii) the new ``incompatibility'' metric represented by
 the superposition of traces of the original incompatibility metric, along
the admissible directions of interaction.

\subsection{The continuum model $\mathcal{E}$}

We now introduce and explain the models involved.
The ``incompatible elasticity'' postulates that the three-dimensional body seeks to realize a configuration
with a prescribed Riemannian metric $G$, and that the resulting
deformation minimizes the energy $\mathcal{E}$ which in turn measures
the deviation of a given deformation from being an
orientation-preserving isometric immersion of $G$.
More precisely,  let $G$ be a smooth Riemannian metric on
an open, bounded, connected domain $\Omega\subset\mathbb{R}^n$,
i.e. $G\in \mathcal{C}^\infty (\bar\Omega, \mathbb{R}^{n\times n})$ and
$G(x)$ is symmetric positive definite for every $x\in \bar\Omega$.  
The shape change that occurs during the growth of $\Omega$ is due to
changes in the local stress-free state (for instance,
material may be added or removed), and to the accommodation of these
changes. Consequently, the gradient of the 
deformation $u: \Omega \rightarrow \mathbb{R}^{n}$ that maps the
original stress-free state to the observed  
state,  can be decomposed as:
\begin{equation*}
 \nabla u = F A,
\end{equation*}
into the growth deformation tensor $A: \Omega \rightarrow
\mathbb{R}^{n \times n}$,  describing the growth from the reference
zero-stress state to a new locally stress-free state, and the 
elastic deformation tensor $F$. The elastic energy $\mathcal{E}$ is
then given in terms of $F$, by:
\begin{equation}\label{E} 
\mathcal{E}(u) = \int_{\Omega} \overline{W}(F)~\mbox{d}x = \int_{\Omega}
\overline{W}(\nabla u A^{-1})~\mbox{d}x. 
\end{equation}
Here, the density potential $\overline{W}: \mathbb{R}^{n \times n}
\rightarrow \overline{\mathbb{R}}_+$ satisfies the following assumptions of frame
invariance with respect to the group of proper rotations $SO(n)$,
normalization, and non-degeneracy:
\begin{equation}\label{assu}
\forall F \in \mathbb{R}^{n \times
  n}, R \in SO(n) \qquad \overline{W}(RF) = \overline{W}(F),\quad \overline{W}(R) = 0,\quad 
\overline{W}(F) \geq c ~\mbox{dist}^{2}(F, SO(n)), 
\end{equation}
for some uniform constant $c >0$.

Observe that: $\mathcal{E}(u)=0$ is equivalent to $\nabla u(x) \in
SO(n)A(x)$ for almost every $x\in\Omega$. Further, in view of the polar
decomposition theorem, the same condition is
equivalent to: $(\nabla u)^{T}\nabla u = A^{T}A$ and $\det \nabla u >
0$ in $\Omega$, i.e. $\mathcal{E}(u)=0$ if and only if $u$ is an isometric immersion of the imposed
Riemannian metric $G = A^{T} A$. Hence, when 
$G$ is not realizable (i.e. when its Riemann curvature tensor does not
vanish identically in $\Omega$), there is no $u$ with
$\mathcal{E}(u)=0$. It has further been proven in \cite{LP} that in this case:
$\inf\{ \mathcal{E}(u); ~ u\in W^{1,2}(\Omega, \mathbb{R}^n\}>0$ as well, which points to the existence of residual
non-zero strain at free equilibria of $\mathcal{E}$.

Given $G$, we will call $A=\sqrt{G}$, and without loss of generality we always
assume that $A$ is symmetric and strictly positive definite in
$\Omega$.

\subsection{The discrete model $E_\epsilon$}

We now describe the discrete model whose asymptotic
behavior we intend to study.
The total stored discrete energy of a given deformation acting on the
atoms of the lattice in $\Omega$, is defined to be 
the superposition of the energies weighting the pairwise
interactions between the atoms, with respect  to
$G$. More precisely, given $\epsilon>0$ and a discrete map
$u_\epsilon:\epsilon\mathbb{Z}^n\cap\Omega \rightarrow \mathbb{R}^n$, let:
\begin{equation}\label{discrete_energy}
E_\epsilon (u_\epsilon) = \sum_{\xi \in \mathbb{Z}^{n}}\sum_{\alpha \in R^{\xi}_\epsilon (\Omega)} 
\epsilon^{n} \psi (\vert \xi \vert) \Big \vert \dfrac{\vert u_\epsilon(\alpha + \epsilon \xi) - 
u_\epsilon (\alpha) \vert}{\epsilon \vert A(\alpha)\xi \vert} - 1 \Big \vert ^{2},
\end{equation}
where $R_\epsilon^{\xi}(\Omega)= \{\alpha \in \epsilon \mathbb{Z}^{n}
: [\alpha, \alpha + \epsilon \xi] \subset \Omega\}$ denotes the set of
lattice points interacting with the node $\alpha$, and
where a smooth cut-off function
$\psi:\mathbb{R}_+\rightarrow \mathbb{R}$ allows only for interactions
with finite range: 
$$\psi(0) = 0 \quad \mbox{ and }\quad \exists M>0 \quad \forall n \geq M \quad \psi(n) = 0.$$
The energy in (\ref{discrete_energy}) measures the discrepancy between
lengths of the actual displacements between the nodes
$x=\alpha+\epsilon\xi$ and $y=\alpha$ due to the deformation $u_\epsilon$,
and the ideal displacement length $\langle G(\alpha) (x-y),
(x-y)\rangle^{1/2} = \epsilon |A(\alpha)\xi|.$
Note that the measure of this dicrepancy in terms of the ratio
$\frac{l}{l_0}$ of the actual length $l = \vert u_\epsilon(\alpha + \epsilon \xi) - 
u_\epsilon (\alpha) \vert $ and the ideal length $l_0 = \epsilon \vert A(\alpha)\xi \vert$ 
is present in the reconstruction of an elastic law for the myocardium from the 
known behavior of the myocytes in \cite{cmr} (formula (11)).

When $\epsilon\to 0$ and when sampling on sufficiently many
interaction directions $\xi$, one might expect that (\ref{discrete_energy}) will
effectively measure the discrepancy between all lengths $|u(x) -
u(y)|$ and the ideal lengths $|A(x) (x-y)|$ determined by the imposed
metric, as in (\ref{E}).
For $G=\mbox{Id}$, it has been proven in \cite{AC} that this is indeed
the case, as well as that the
$\Gamma$-limit $\mathcal{F}$ of $E_\epsilon$ has the form: $\mathcal{F}(u)
= \int_\Omega f(\nabla u)~\mbox{d}x$ with the limiting density $f$ frame invariant and quasiconvex.

\subsection{The main results and the organization of the paper}

Towards studying the energies (\ref{discrete_energy}),  we first derive an 
integral representation for $E_\epsilon$ by introducing a family of lattices determined by each
length of the admissible interactions (when $\psi\neq 0$); this is
done in sections \ref{jeden} and \ref{rep3}. Since the general formula
for the integral representation 
uses quite involved notation, we
first present its simpler versions, valid in cases of the nearest-neighbour and
next-to-nearest-neighbour interactions. For each lattice, we define its
$n$-dimensional triangulation and, as usual in the lattice analysis, we
associate with it the piecewise affine maps matching with the original discrete
deformations at each node. 

In section \ref{sec4} we derive the lower
and upper bounds $I_Q$ and $I$ of the $\Gamma$-limit $\mathcal{F}$ of $E_\epsilon$, as $\epsilon\to
0$, in terms of the superposition of integral energies defined
effectively on the $W^{1,2}$ deformations of $\Omega$.
The disparity between the upper and lower bounds reflects the fact
that each lattice in the discrete description gives rise, in general, to a
distinct recovery sequence of the associated $\Gamma$-limit.
This is hardly surprising, since the operation of taking the lsc envelope of an integral energy is not
additive (nor is the operation of quasiconvexifcation of its density).

On the other hand, each term in $I_Q$ and $I$ has the structure as in
(\ref{E}), but with $G$ replaced by other effective metric induced 
by the distinct lattices. In case of only nearest-neighbour or
next-to-nearest-neighbour interactions all the effective metrics coincide with one
residual metric $\bar G$. This further
allows to obtain the formula for $\mathcal{F}$, which is accomplished in
section \ref{sec5}. In section \ref{sec6} we compare $\mathcal{F}$
with $\mathcal{E}$ through a series of examples. We note, in particular, that the realisability of $G$
does not imply the realisability of $\bar G$, neither the converse of
this statement is true.

Finally, in the Appendix section \ref{secap} we gather some
classical facts on $\Gamma$ convergence and convexity, which we use in
the proofs of this note.

\medskip

Let us conclude by remarking  that a continuum finite range interaction model,
in the spirit of (\ref{discrete_energy}), can be posed similarly to the
models considered recently in \cite{b1, b2, m}, by:
$$\tilde E_\epsilon(u) = \int_\Omega\int_\Omega \psi\Big(\frac{|x-y|}{\epsilon}\Big)
\left|\frac{|u(x) - u(y)|}{|A(x)(x-y)|} -1\right|^2~\mbox{d}x\mbox{d}y.$$
It would be interesting to find the $\Gamma$-limit of $\tilde
E_\epsilon$, as $\epsilon\to 0$ and compare it with both $\mathcal{E}$
and $\mathcal{F}$.

\subsection{Notation}
Throughout the paper, $\Omega$ is an open bounded subset of
$\mathbb{R}^n$. For $s>0$, we denote:
$$ \Omega_s = \{x \in \Omega; ~~ \mbox{dist}(x, \partial \Omega) > s\}.$$

The standard triangulation of the $n$-dimensional cube $C_n = [0, 1]^{n}$ 
is defined as follows.
For all permutations $\pi\in S_n$ of $n$ elements, let $T^\pi$ be the $n$-simplex obtained by:
\begin{equation*}
T^\pi = \{(x_1, ..., x_n)\in C_n; ~~ x_{\pi(1)} \geq \cdots \geq x_{\pi(n)}\}.
\end{equation*}
Note that $T^\pi$ is the convexification of its vertices:
\begin{eqnarray*}
T^\pi = \mbox{conv} \Big\{0, e_{\pi(1)}, e_{\pi(1)} + e_{\pi(2)},\ldots, e_{\pi(1)} + \cdots +  e_{\pi(n)} = e_1+ \cdots + e_n\Big\},
\end{eqnarray*} 
and that all simplices $T^\pi$ have $0$ and $(1,\ldots,1)=e_1+\ldots +
e_n$ as common vertices.
The collection of $n!$ simplices $\{T^\pi\}_{\pi\in S_n}$ constitutes the standard triangulation of $C_n$,
which can also be naturally extended to each cell $\alpha+\epsilon
C_n$ where $\alpha\in\epsilon\mathbb{Z}^n$:
$$ T_\alpha^\pi = \mbox{conv} \Big\{\alpha, \big\{\alpha + \epsilon \sum_{i=1}^je_{\pi(i)}\big\}_{j=1}^n\Big\}.  $$
When $\pi=(i_1,\ldots ,i_n)$ we shall also write
$T_\alpha^{(i_1,\ldots, i_n)} = T_\alpha^\pi = \mbox{conv}
\Big\{\alpha, \big\{\alpha + \epsilon
\sum_{k=1}^je_{i_k}\big\}_{j=1}^n\Big\}.  $
Moreover, we call:
\begin{equation}\label{form}
\mathcal{T}_{\epsilon, n} = \{T_\alpha^\pi; ~~ \alpha\in\epsilon\mathbb{Z}^n, ~ \pi\in S_n\}.
\end{equation}

Finally, by $C$ we denote any universal constant, depending on $\Omega$ and $W$,
but independent of other involved quantities at hand.

\bigskip

\noindent{\bf Acknowledgments.} 
M.L. was partially supported by the NSF Career grant DMS-0846996.

\section{Integral representation of discrete energies
  (\ref{discrete_energy}) - special cases}\label{jeden}

Since the general formula for the integral representation of 
$E_\epsilon$, given in section \ref{rep3}, uses a somewhat involved
notation which may obscure the construction, we
first present its simpler versions, valid in cases of the near and
next-to-nearest-neighbour interactions, which we further discuss in sections
\ref{sec5} and \ref{sec6}.

\subsection{Case 1: nearest-neighbour interactions in $\mathbb{R}^2$}
Let $\Omega\subset\mathbb{R}^2$ and assume that $\psi(1)=1$ and $\psi(|\xi|) =0$ for $|\xi|\geq\sqrt{2}$.
The energy (\ref{discrete_energy}) of a  deformation 
$u_\epsilon: \epsilon \mathbb{Z}^{2} \cap \Omega \rightarrow
\mathbb{R}^{2}$,  takes then the form:
\begin{eqnarray*}
E_\epsilon (u_\epsilon) = \sum_{i, j = 1}^{2}\sum_{\alpha \in
  R^{(-1)^{j}e_i}_\epsilon (\Omega)} \epsilon^{2}\Big \vert
\dfrac{\vert u_\epsilon(\alpha + (-1)^{j}\epsilon e_i) - u_\epsilon
  (\alpha) \vert}{\epsilon \vert A(\alpha)e_i \vert} - 1 \Big \vert
^{2}. 
\end{eqnarray*} 
Let $U_\epsilon\subset \Omega$ be the union of those (open) cells in the lattice $\epsilon \mathbb{Z}^2$,
 which have non-empty intersection with the set $\Omega_{\sqrt{2} \epsilon}$.
We consider the standard triangulation $\mathcal{T}_{\epsilon, 2}$ of the lattice $\epsilon \mathbb{Z}^{2}$, 
as in (\ref{form}), and we  identify the discrete map $u_\epsilon$ with the unique continuous function on $U_\epsilon$, 
affine on all the triangles in $\mathcal{T}_{\epsilon, 2}\cap U_\epsilon$, and matching with $u_\epsilon$
at each node.

Define the function $W: \mathbb{R}^{2 \times 2} \rightarrow \mathbb{R}$:
\begin{eqnarray*}
W([M_{ij}]_{i,j=1..2}) = \sum_{j=1}^2 \left(\Big(\sum_{i=1}^2
  |M_{ij}|^2\Big)^{1/2} - 1\right)^2 \qquad\forall
M\in\mathbb{R}^{2\times 2}.
\end{eqnarray*} 
We easily see that for every $\alpha\in\epsilon\mathbb{Z}^2\cap U_\epsilon$:
\begin{equation*}
\begin{split}
&\epsilon^{2}\left( \Big|\frac{|u_\epsilon(\alpha + \epsilon e_1) -
    u_\epsilon(\alpha)|}{\epsilon| A(\alpha) e_1|} -1\Big|^2 +
  \Big|\frac{|u_\epsilon(\alpha + \epsilon (e_1 + e_2)) -
    u_\epsilon(\alpha + \epsilon e_1)|}{\epsilon| A(\alpha + \epsilon
    e_1) e_2 |} - 1 \Big|^2\right) \\ & \qquad\qquad 
=  2\int_{T_\alpha^{(1,2)}}W(\nabla u_\epsilon(x)\lambda_\epsilon(x))~\mbox{d}x,
\end{split}
\end{equation*}
where $\lambda_\epsilon:U_\epsilon\rightarrow\mathbb{R}^{2\times 2}$ is a piecewise constant matrix field, given by:
\begin{equation*}
\begin{split}
\forall x\in T_\alpha^{(1,2)}\cap U_\epsilon,\qquad & \lambda_\epsilon(x) = \mbox{diag}\left\{
\vert A(\alpha)e_1 \vert^{-1}, \vert A(\alpha+\epsilon e_1)e_2 \vert^{-1} \right\} \\
\forall x\in T_\alpha^{(2,1)}\cap U_\epsilon,\qquad & \lambda_\epsilon(x) = \mbox{diag}\left\{
\vert A(\alpha+\epsilon e_2)e_1 \vert^{-1} , \vert A(\alpha)e_2 \vert^{-1}\right\}
\end{split}
\end{equation*}
while we recall that $T_\alpha^{(1,2)}= \mbox{conv}\{\alpha, \alpha+\epsilon e_1, \alpha+\epsilon(e_1+e_2)\}$ and 
$T_\alpha^{(2,1)}= \mbox{conv}\{\alpha, \alpha+\epsilon e_2, \alpha+\epsilon(e_1+e_2)\}$.
Similarly, we get:
\begin{equation*}
\begin{split}
&\epsilon^{2}\left( \Big|\frac{|u_\epsilon(\alpha + \epsilon e_2) -
    u_\epsilon(\alpha)|}{\epsilon| A(\alpha) e_2|} -1\Big|^2 +
  \Big|\frac{|u_\epsilon(\alpha + \epsilon (e_1 + e_2)) -
    u_\epsilon(\alpha + \epsilon e_2)|}{\epsilon| A(\alpha+\epsilon
    e_2) e_1|}  
- 1 \Big|^2\right) \\ & \qquad\qquad 
=  2\int_{T_\alpha^{(2,1)}}W(\nabla u_\epsilon(x)\lambda_\epsilon(x))~\mbox{d}x.
\end{split}
\end{equation*}

For the interactions in the opposite directions: $-e_1$ and $-e_2$, we obtain:
\begin{equation*}
\begin{split}
&\epsilon^{2}\left( \Big|\frac{|u_\epsilon(\alpha + \epsilon e_1) -
    u_\epsilon(\alpha+\epsilon(e_1+ e_2))|}{\epsilon|
    A(\alpha+\epsilon(e_1+ e_2)) e_2|} -1\Big|^2 +
  \Big|\frac{|u_\epsilon(\alpha) - u_\epsilon(\alpha + \epsilon
    e_1)|}{\epsilon| A(\alpha + \epsilon e_1) e_1 |} - 1
  \Big|^2\right) \\ & \qquad\qquad 
=  2\int_{T_\alpha^{(1,2)}}W(\nabla u_\epsilon(x)\bar\lambda_\epsilon(x))~\mbox{d}x,
\end{split}
\end{equation*}
and:
\begin{equation*}
\begin{split}
&\epsilon^{2}\left( \Big|\frac{|u_\epsilon(\alpha + \epsilon e_2) -
    u_\epsilon(\alpha+\epsilon(e_1+e_2))|}{\epsilon|
    A(\alpha+\epsilon(e_1+e_2)) e_1|} -1\Big|^2 +
  \Big|\frac{|u_\epsilon(\alpha)- u_\epsilon(\alpha + \epsilon
    e_2)|}{\epsilon| A(\alpha + \epsilon e_2) e_2 |} - 1
  \Big|^2\right) \\ & \qquad\qquad 
=  2\int_{T_\alpha^{(2,1)}}W(\nabla u_\epsilon(x)\bar\lambda_\epsilon(x))~\mbox{d}x,
\end{split}
\end{equation*}
where $\bar\lambda_\epsilon:U_\epsilon\rightarrow\mathbb{R}^{2\times 2}$ is given by:
\begin{equation*}
\begin{split}
\forall x\in T_\alpha^{(1,2)}\cap U_\epsilon,\qquad & \bar\lambda_\epsilon(x) = \mbox{diag}\left\{
\vert A(\alpha+\epsilon e_1)e_1 \vert^{-1},
\vert A(\alpha+\epsilon (e_1+ e_2))e_2 \vert^{-1}\right\}\\
\forall x\in T_\alpha^{(2,1)}\cap U_\epsilon,\qquad & \bar\lambda_\epsilon(x) = \mbox{diag}\left\{
\vert A(\alpha+\epsilon (e_1 +e_2))e_1 \vert^{-1}, \vert A(\alpha+\epsilon e_2)e_2 \vert^{-1}\right\}
\end{split}
\end{equation*}

Summing over all $2$-simplices and noting that each interaction was counted twice, we obtain:
\begin{equation}\label{estim_near_2}
\begin{split}
0& \leq E_\epsilon(u_\epsilon) - I_{\epsilon,1}(u_\epsilon) 
 \leq \sum_{i,j=1}^2 \sum_{ \alpha \in R^{(-1)^je_i}_\epsilon (\overline{\Omega \setminus U_\epsilon})} 
\epsilon^{2} \Big \vert \dfrac{ \vert u_\epsilon(\alpha + \epsilon (-1)^j\epsilon e_i) 
- u_\epsilon (\alpha) \vert}{\epsilon \vert A(\alpha)e_i \vert} - 1 \Big \vert ^{2},
\end{split}
\end{equation}
where:
\begin{equation}\label{I1d}
I_{\epsilon,1}(u_\epsilon)  = \int_{U_\epsilon}\Big( W(\nabla u_\epsilon(x) \lambda_\epsilon(x)) 
+ W(\nabla u_\epsilon(x) \overline{\lambda}_\epsilon(x)) \Big)~ \mbox{d}x. 
\end{equation}

\subsection{Case 2: nearest-neighbour interactions in $\mathbb{R}^{n}$}\label{rep2}

Let now  $\Omega\subset\mathbb{R}^{n}$,
and assume that $\psi(1)=1$ and $\psi(|\xi|)=0$ for $|\xi|\geq \sqrt{n}$.
For small $\epsilon > 0$, define $U_\epsilon\subset \Omega$ as the union of all  cells in $\epsilon\mathbb{Z}^n$, 
with the standard triangulation $\mathcal{T}_{\epsilon, n}$, that have
nonempty intersection with $\Omega_{\epsilon \sqrt{n}}$.
As in Case 1, we identify the given discrete deformation 
$u_\epsilon:\epsilon\mathbb{Z}^n\cap\Omega\rightarrow \mathbb{R}^n$ with its unique extension to 
the continuous function on $U_\epsilon$, affine on all of the
$n$-dimensional simplices in $\mathcal{T}_{\epsilon, n}\cap U_\epsilon$.

We also have  $W: \mathbb{R}^{n \times n} \rightarrow \mathbb{R}$:
\begin{equation}\label{W}
 W([M_{ij}]_{i,j:1..n})= \sum_{i = 1}^{n}
\left((\sum_{i=1}^n|M_{ij}|^2)^{1/2} - 1\right)^{2} \qquad \forall
M\in\mathbb{R}^{n\times n}.
\end{equation}
Note that for any permutation $\pi\in S_n$ one has:
\begin{equation*}
\begin{split}
&\epsilon^{n}\sum_{j=0}^{n-1}\Big|\frac{|u_\epsilon(\alpha + \epsilon \sum_{i=1}^{j+1} e_{\pi(i)}) 
- u_\epsilon(\alpha + \epsilon \sum_{i=1}^{j} e_{\pi(i)})|}{\epsilon| A(\alpha + \epsilon \sum_{i=1}^{j} e_{\pi(i)}) e_{\pi(j+1)}|}
 -1\Big|^2    \\ & \qquad\qquad
=  n!\int_{T_\alpha^{\pi}}W(\nabla u_{\epsilon}(x)\lambda_{\epsilon}(x))~\mbox{d}x,
\end{split}
\end{equation*}
where the piecewise constant matrix field $\lambda_{\epsilon}$ is given by:
\begin{equation}\label{for1}
\forall x\in T_\alpha^{\pi}\cap U_{\epsilon},\qquad \lambda_{\epsilon}(x) = \mbox{diag} \left\{
\vert A(\alpha + \epsilon \sum_{i=1}^{\pi^{-1}(j)-1} e_{\pi(i)}) e_j \vert^{-1}\right\}_{j=1}^n. 
\end{equation}

To include the interactions in $\{-e_i\}$ directions, as before, we write:
\begin{equation*}
\begin{split}
&\epsilon^{n}\sum_{j=0}^{n-1}\Big|\frac{|u_\epsilon(\alpha + \epsilon
  \sum_{i=1}^{j} e_{\pi(i)}) - u_\epsilon(\alpha + \epsilon
  \sum_{i=1}^{j+1} e_{\pi(i)})|}{\epsilon| A(\alpha + \epsilon
  \sum_{i=1}^{j+1} e_{\pi(i)}) e_{\pi(j+1)}|} -1\Big|^2    \\ &
\qquad\qquad 
=  n!\int_{T_\alpha^{\pi}}W(\nabla u_{\epsilon}(x)\bar\lambda_{\epsilon}(x))~\mbox{d}x,
\end{split}
\end{equation*}
where: 
\begin{equation}\label{for2}
\forall x\in T_\alpha^{\pi}\cap U_{\epsilon},\qquad \bar\lambda_{\epsilon}(x) = \mbox{diag} \left\{
\vert A(\alpha + \epsilon \sum_{i=1}^{\pi^{-1}(j)} e_{\pi(i)}) e_j \vert^{-1}\right\}_{j=1}^n. 
\end{equation}

Summing over all of the $n$-simplices, and noting that each one-length
interaction is counted $n!$ times, we obtain: 
\begin{equation}\label{estim2}
\begin{split}
0 & \leq E_\epsilon(u_\epsilon) - I_{\epsilon, 1}(u_\epsilon) \leq\sum_{\vert \xi \vert
  = 1}  
\sum_{ \alpha \in R^{\xi}_\epsilon (\overline{\Omega \setminus
    U_\epsilon})} \epsilon^{n} \Big \vert \dfrac{ \vert
  u_\epsilon(\alpha + \epsilon \xi) - u_\epsilon (\alpha)
  \vert}{\epsilon \vert A(\alpha)\xi \vert} - 1 \Big \vert ^{2},
\end{split}
\end{equation}
where $I_{\epsilon, 1}$ is given by the same formula as in
(\ref{I1d}), with $\lambda_\epsilon$ and $\bar\lambda_\epsilon$
defined as in (\ref{for1}), (\ref{for2}).

\subsection{Case 3: next-to-nearest-neighbour interactions in $\mathbb{R}^2$}\label{sec3}

Let us assume now again that $\Omega\subset\mathbb{R}^2$, and that 
$\psi(\sqrt{2})=1$ and $\psi(|\xi|)=0$ for $|\xi|\geq \sqrt{3}$ and
$|\xi|\leq 1$.
Our goal now is to obtain a similar representation and  bound to
(\ref{estim_near_2}) (\ref{I1d}) for the discrete energy 
corresponding to the next-to-nearest-neighbour interactions of length $\sqrt{2}$.
The canonical lattice $\epsilon\mathbb{Z}^{2}$ is now mapped onto the lattice $\epsilon B\mathbb{Z}^2$, where:
$$ B = \left[\begin{array}{cc}
1 & -1 \\  1 & 1
\end{array}\right]. $$
We will also need to work with the translated lattice  $\epsilon(e_1 + B\mathbb{Z}^2)$.
Let $U^0_{\epsilon, \sqrt{2}}\subset\Omega$ be the union of all open
cells in the lattice $\epsilon B\mathbb{Z}^2$ which have nonempty
intersection with $\Omega_{2\epsilon}$. Define
$u^0_{\epsilon,\sqrt{2}}$ to be the unique continuous function on
$U^0_{\epsilon,\sqrt{2}}$, affine on the triangles of the induced
triangulation $B\mathcal{T}_{\epsilon, 2}\cap U_{\epsilon,\sqrt{2}}^0$, matching with the original
deformation $u_\epsilon$ at each node of the lattice $\epsilon B
\mathbb{Z}^2\cap U^0_{\epsilon, \sqrt{2}}$. 
Likewise, by $U^1_{\epsilon, \sqrt{2}}\subset\Omega$ we call the union
of cells in the lattice $\epsilon (e_1 + B\mathbb{Z}^2)$ which have
nonempty intersection with $\Omega_{2\epsilon}$, while
$u^1_{\epsilon,\sqrt{2}}$ is the matching continuous piecewise affine (on
triangles in $\epsilon e_1 + B\mathcal{T}_{\epsilon, 2}$) extension of
$u_\epsilon$. 

Denoting $\xi_1 = Be_1$ and $\xi_2=Be_2$ we obtain, as before:
\begin{equation*}
\begin{split}
&\epsilon^{2}\Big( \Big|\frac{|u_\epsilon(B(\alpha + \epsilon e_1)) 
- u_\epsilon(B\alpha)|}{\epsilon| A(B\alpha) \xi_1|} -1\Big|^2  
+   \Big|\frac{|u_\epsilon(B(\alpha +
  \epsilon(e_1+e_2))) - u_\epsilon(B(\alpha + \epsilon
  e_1))|}{\epsilon| A(B(\alpha + \epsilon e_1)) \xi_2 |} - 1
\Big|^2\Big) \\ & \qquad\qquad\qquad\qquad\qquad 
=  \frac{2}{|\det B|}\int_{BT_\alpha^{(1,2)}}W(\nabla
u^0_{\epsilon,\sqrt{2}}(x)\lambda^0_{\epsilon,\sqrt{2}}(x))~\mbox{d}x, 
\end{split}
\end{equation*}
where $\lambda^0_{\epsilon,\sqrt{2}}:U^0_{\epsilon,\sqrt{2}}\rightarrow\mathbb{R}^{2\times 2}$ is given by:
\begin{equation*}
\begin{split}
\forall x\in BT_\alpha^{(1,2)}\cap U^1_{\epsilon,\sqrt{2}}\qquad & \lambda^0_{\epsilon,\sqrt{2}}(x) = \sqrt{2} B 
\mbox{diag}\left\{
\vert A(B\alpha)\xi_1 \vert^{-1}, \vert A(B(\alpha+\epsilon e_1))\xi_2 \vert^{-1}  \right\}\\
\forall x\in BT_\alpha^{(2,1)}\cap U^1_{\epsilon,\sqrt{2}}\qquad & \lambda^0_{\epsilon,\sqrt{2}}(x) = \sqrt{2} B
\mbox{diag} \left\{\vert A(B(\alpha+\epsilon e_2))\xi_1 \vert^{-1}, \vert A(B\alpha)\xi_2 \vert^{-1}\right\}.
\end{split}
\end{equation*}
Interactions in the opposite directions $-\xi_i$, yield the integrals:
$$ \frac{2}{\vert \det B\vert}\int_{BT_\alpha^{1,2}}W(\nabla u^0_{\epsilon,\sqrt{2}}(x) 
\bar\lambda^0_{\epsilon,\sqrt{2}}(x)) ~ \mbox{d}x, $$
where now $\bar\lambda^0_{\epsilon,\sqrt{2}}:U^1_{\epsilon,\sqrt{2}}\rightarrow\mathbb{R}^{2\times 2}$ satisfies:
\begin{equation*}
\begin{split}
\forall x&\in B T_\alpha^{(1,2)} \cap U^1_{\epsilon,\sqrt{2}}\qquad \\
& \bar\lambda^0_{\epsilon,\sqrt{2}}(x) = \sqrt{2}B \mbox{ diag}\left\{
\vert A(B(\alpha+\epsilon e_1))\xi_1 \vert^{-1}, \vert A(B(\alpha+\epsilon (e_1+e_2)))\xi_2 \vert^{-1} \right\},
\\
\forall x&\in B T_\alpha^{(2,1)}\cap U^1_{\epsilon,\sqrt{2}}\qquad \\
& \bar\lambda^0_{\epsilon,\sqrt{2}}(x) = \sqrt{2}B \mbox{ diag}\left\{
\vert A(B(\alpha+\epsilon (e_1+e_2)))\xi_1 \vert^{-1}, \vert A(B(\alpha+\epsilon e_2))\xi_2 \vert^{-1} \right\}.
\end{split}
\end{equation*}
Similarly, we obtain the integral representations on the triangulation 
$\epsilon e_1 + B\mathcal{T}_{\epsilon, 2}$ of the set
$U^1_{\epsilon,\sqrt{2}}$:
$$ \int W(\nabla u^1_{\epsilon,\sqrt{2}}(x)
\lambda^1_{\epsilon,\sqrt{2}}(x)) ~ \mbox{d}x \qquad \mbox{and} \qquad
\int W(\nabla u^1_{\epsilon,\sqrt{2}}(x) \lambda^1_{\epsilon,\sqrt{2}}(x)) ~ \mbox{d}x,$$ 
with the piecewise affine functions:
\begin{equation*}
\begin{split}
& \forall x\in (\epsilon e_1 + BT_\alpha^{(1,2)})\cap U^1_{\epsilon,\sqrt{2}} \qquad \\
& \qquad \lambda^1_{\epsilon,\sqrt{2}}(x) = \sqrt{2}B 
\mbox{ diag}\left\{
\vert A(\epsilon e_1 + B\alpha)\xi_1 \vert^{-1}, \vert A(\epsilon e_1 + B(\alpha+\epsilon e_1))\xi_2 \vert^{-1} 
\right\} \\
& \forall x\in (\epsilon e_1 + BT_\alpha^{(2,1)})\cap U^1_{\epsilon,\sqrt{2}} \qquad \\
& \qquad \lambda^1_{\epsilon,\sqrt{2}}(x) =\sqrt{2}B
\mbox{ diag}\left\{
\vert A(\epsilon e_1 + B(\alpha+\epsilon e_2))\xi_1 \vert^{-1}, \vert A(\epsilon e_1 + B\alpha)\xi_2 \vert^{-1}
\right\} \\
& \forall x\in (\epsilon e_1 + BT_\alpha^{(1,2)})\cap
U^1_{\epsilon,\sqrt{2}} \quad \\ &
\qquad \bar\lambda^1_{\epsilon,\sqrt{2}}(x)  =
\sqrt{2}B \mbox{ diag}\left\{
\vert A(\epsilon e_1 + B(\alpha+\epsilon e_1))\xi_1 \vert^{-1}, 
\vert A(\epsilon e_1 + B(\alpha+\epsilon (e_1+e_2)))\xi_2 \vert^{-1}
\right\} \\
& \forall x\in (\epsilon e_1 + BT_\alpha^{(2,1)})\cap U^2_{\epsilon,\sqrt{2}} \quad \\
& \qquad \bar\lambda^1_{\epsilon,\sqrt{2}}(x) = \sqrt{2}B \mbox{ diag}\left\{ 
\vert A(\epsilon e_1 + B(\alpha+\epsilon (e_1+e_2)))\xi_1 \vert^{-1}, 
\vert A(\epsilon e_1 + B(\alpha+\epsilon e_2))\xi_2 \vert^{-1}
\right\}
\end{split}
\end{equation*}
Consequently:
\begin{equation}\label{estim_next_near_2}
\begin{split}
0 & \leq E_\epsilon(u_\epsilon)-I_{\epsilon,\sqrt{2}}(u_\epsilon) \\ & \leq 
\sum_{i,j=1}^2 \sum_{ \alpha \in R_\epsilon^{(-1)^j\xi_i} (\overline{\Omega \setminus \Omega_{2\epsilon}})} 
\epsilon^{2} \Big \vert \frac{ \vert u_\epsilon(\alpha + \epsilon (-1)^j\xi_i) 
- u_\epsilon (\alpha) \vert}{\epsilon \vert A(\alpha)\xi_i \vert} - 1 \Big \vert ^{2},
\end{split}
\end{equation}
where:
\begin{equation*}
\begin{split}
I_{\epsilon,\sqrt{2}}(u_\epsilon)  = ~ & \frac{1}{2}\int_{U^0_{\epsilon, \sqrt{2}}}
\Big(W(\nabla u^0_{\epsilon, \sqrt{2}}(x)\lambda^0_{\epsilon, \sqrt{2}}(x))  
+ W(\nabla u^1_{\epsilon,\sqrt{2}}(x)\bar \lambda_{\epsilon, \sqrt{2}}^1(x))\Big)~\mbox{d}x \\ & +
\frac{1}{2}\int_{U^1_{\epsilon, \sqrt{2}}}\Big(W(\nabla u^1_{\epsilon, \sqrt{2}}(x)\lambda^1_{\epsilon, \sqrt{2}}(x))  
+ W(\nabla u^1_{\epsilon,\sqrt{2}}(x)\bar \lambda_{\epsilon, \sqrt{2}}^1(x))\Big)~\mbox{d}x.
\end{split}
\end{equation*} 

\section{Integral representation of discrete energies
  (\ref{discrete_energy}) - the general case}\label{rep3}

\begin{lemma}\label{algo}
Let $\xi = (\xi^1,\ldots, \xi^n)\in \mathbb{Z}^n\setminus \{0\}$. Let
$k$ denote the number of non-zero coordinates in $\xi$, and denote:
$\xi^{i_1},\ldots, \xi^{i_k} $ $\neq 0$ with $i_1<i_2\ldots<i_k$,
while $\xi^{j_1} =\ldots = \xi^{j_{n-k}} = 0$ with
$j_1<j_2\ldots<j_{n-k}$.
Fix $\bar s\in\{1\ldots k\}$ and define $n$ vectors $\xi_1,\ldots,
\xi_n\in\mathbb{Z}^n$ by the following algorithm:
\begin{equation*}
\begin{split}
& \xi_1 = \xi\\
& \forall p=2,\ldots, k-\bar s+1 \qquad \xi_p^{i_{\bar s - 1 +p}} = -
\xi^{i_{\bar s - 1 +p}},  ~~\mbox{ and } ~~\xi_p^i = \xi^i ~\mbox{ for all
  other indices } i \\
& \forall p=k-\bar s+2,\ldots, k\qquad \xi_p^{i_{\bar s - 1 +p-k}} = -
\xi^{i_{\bar s - 1 +p-k}},  ~~\mbox{ and } ~~\xi_p^i = \xi^i ~\mbox{ for all
  other indices } i \\
& \forall p=k+1,\ldots, n\qquad \xi_p^{i_{\bar s}} = 0, ~~ 
\xi_p^{j_{p-k}} = \xi^{i_{\bar s}},  ~~\mbox{ and } ~~\xi_p^i = \xi^i ~\mbox{ for all
  other indices } i.
\end{split}
\end{equation*}
(In other words, given $\xi$ and fixing one of its nonzero coordinates
$i_{\bar s}$, we first change sign of all its nonzero coordinates but
$\xi^{i_{\bar s}}$, in the cyclic order, starting from  $\xi^{i_{\bar
    s}}$: this gives $k$ vectors $\xi_p$. Then we permute the
$\xi^{i_{\bar s}}$ coordinate with all the zero coordinates: this
gives the remaining $n-k$ coordinates).

Then the $n$-tuple of vectors $\xi_1,\ldots, \xi_n$ is linearly independent.
\end{lemma}
\begin{proof}
Without loss of generality, we may assume that $i_p=p$ for all $
p=1,\ldots, k$ and $\bar s =1$.

Consider first the case when $k=n$, i.e. when all coordinates of the
vector $\xi$ are nonzero.
Then the matrix $B=[\xi_1,\ldots, \xi_n]$ is similar to the following
matrix:
$$\tilde B = \left[\begin{array}{ccccc} 1 & 1 & 1 & \cdots & 1\\
1 & -1 & 1 & \cdots & 1\\ 1 & 1 & -1 & \cdots & 1\\
\cdots & \cdots & \cdots & \cdots & \cdots \\
 1 & 1 & 1 & \cdots & -1
 \end{array}\right],$$
by the basic operations of dividing each row by $|\xi^i|$. The matrix $\tilde B$
above has nonzero determinant, which proves the claim.

Assume now that $k\neq n$, i.e. the last $n-k>0$ coordinates of $\xi$
are zero. Then, the $k\times k$ principal minor of the matrix
$B=[\xi_1,\ldots, \xi_n]$ is invertible, as in the first case
above. The minor consisting of $n-k$ last rows and $k$ first columns
of $B$ equals zero, hence $B$ is invertibe if and only if its minor $B_0$
consisting of $n-k$ last rows and $n-k$ last columns is
invertible. But $B_0 = \xi^{i_{\bar s}}\mbox{Id}_{n-k}$ and hence the
lemma is achieved.
\end{proof}

\subsection{Case 4:  interactions of  a given length $|\xi_0| \neq 0$ in $\mathbb{R}^n$}\label{sec44}
Assume now that $\Omega\subset\mathbb{R}^n$ and let 
$\psi(|\xi_0|)=1$ and $\psi(|\xi|)=0$ for $||\xi| - |\xi_0||>s$, and a
small $s>0$.
Consider the following set of unordered $n$-tuples, which we assume to be nonempty:
\begin{equation} 
\label{sxi}
S_{\vert \xi_0 \vert} = \Big\{ \zeta=\{\zeta^1, ..., \zeta^n\} \subset \mathbb{Z}, 
 ~~  |\zeta|^{2} = \vert \xi_0 \vert^2\Big\}. 
\end{equation}
Fix $\zeta\in S_{|\xi_0|}$ and let $N_\zeta$ be the set of all
distinct signed permutations without repetitions of the coordinates of
$\zeta$, i.e.:
\begin{equation}\label{try}
N_\zeta = \left\{(\pm \zeta^{\pi(1)}, \pm \zeta^{\pi(2)}, \ldots, \pm
  \zeta^{\pi(n)}); ~ \pi\in S_n\right\}.
\end{equation}
Clearly: $|N_\zeta| = 2^k \frac{n!}{k_1!\ldots k_n!}$, where
$k_1,\ldots, k_n$ denote the numbers of repetitions of distinct
coordinates in $\zeta$, and $k$ is the number of non-zero coordinates
in $\zeta$.

For each $\xi\in N_\zeta$ and each of its $k$ non-zero entries
$\xi^{i_{\bar s}}$ we define the set of linearly independent vectors
$\xi_1,\ldots \xi_n$ using the algorithm described in Lemma
\ref{algo}. We call $K_\zeta$ the set of all matrices $B=[\xi_1,\ldots \xi_n]$ obtained by this
procedure; it corresponds to the set of lattices $\epsilon B
\mathbb{Z}^n$ whose edges have lengths $\epsilon|\xi_0|$. Note that:
$$|K_\zeta| = k|N_\zeta| = 2^k k \frac{n!}{k_1!\ldots k_n!}.$$

\begin{lemma}\label{multiplicity}
Let $\zeta\in S_{|\xi_0|}$ have $k$ non-zero entries. Then every
vector $\xi\in N_\zeta$ is included in exactly $nk$ lattices $B$, as
described above.
\end{lemma}
\begin{proof}
Firstly, the number of lattices where $\xi$ is one of the first $k$
columns of $B$, equals $k^2$ ($k$ possible columns and $k$ choices of
a non-zero entry $\xi^{i_{\bar s}}$). Secondly, the number of lattices
where $\xi$ is one of the last $n-k$ columns, equals $(n-k)k$ (given
by $n-k$
possible columns and $k$ choices of a non-zero entry which defines the
first vector in $B$). We hence obtain $nk$ total number of lattices,
as claimed.
\end{proof}

\begin{remark}
The total number of vectors (with repetitions) which are 
columns of lattices in the set $K_\zeta$, equals $|K_\zeta| n = nk
|N_\zeta|$. This is consistent with Lemma \ref{multiplicity}, as each
vector in $N_\zeta$ is repeated $nk$ times.
\end{remark}

\medskip

We now construct the integral representation of the discrete energy in
the presently studied Case 4.
Fix $B\in K_\zeta$ as above, and define $U^{0, B}_{\epsilon, |\xi_0|}\subset \Omega$ to be the union 
of all open cells in $\epsilon B \mathbb{Z}^n$ that have nonempty
intersection with $\Omega_{\epsilon \sqrt{n} |\xi_0|}$. 
We identify the discrete deformation $u_\epsilon$ with its unique continuous extension $u^{0, B}_{\epsilon, |\xi_0|}$ 
on $U^{0, B}_{\epsilon, |\xi_0|}$, affine on all the simplices of the induced triangulation $\epsilon B\mathcal{T}_{\epsilon, n}$.
Following the same observations as in the particular cases before, we
obtain, for any $\pi\in S(n)$:
\begin{equation*}
\begin{split}
&\epsilon^{n}\sum_{j=0}^{n-1}\Big|\frac{|u_\epsilon(B(\alpha + \epsilon \sum_{i=1}^{j+1} e_{\pi(i)})) 
- u_\epsilon(B(\alpha + \epsilon \sum_{i=1}^{j} e_{\pi(i)}))|}
{\epsilon| A(B(\alpha + \epsilon \sum_{i=1}^{j} e_{\pi(i)})) e_{\pi(j+1)}|} -1\Big|^2    
\\ & \qquad\qquad \qquad\qquad
=  \frac{n!}{|\det B|}\int_{BT_\alpha^{\pi}}W(\nabla u^{0, B}_{\epsilon, |\xi_0|}(x)\lambda^{0, B}_{\epsilon,|\xi_0|}(x))~\mbox{d}x,
\end{split}
\end{equation*}
where $W$ is as in (\ref{W}), and:
\begin{equation*}
\begin{split}
& \forall x\in BT_\alpha^{\pi}\cap U^{0, B}_{\epsilon, |\xi_0|} \\ & \qquad 
\lambda^{0, B}_{\epsilon, |\xi_0|}(x) = |\xi_0| B\mbox{ diag} \left\{
\vert A(B(\alpha + \epsilon \sum_{i=1}^{\pi^{-1}(j)-1} e_{\pi(i)})) Be_j \vert^{-1}\right\}_{j=1}^n. 
\end{split}
\end{equation*}
In order to take into account all of the interactions of length $\vert \xi_0 \vert$, we need to consider traslations of
the lattice $\epsilon B \mathbb{Z}^n $. Define:
\begin{equation}\label{VB}
V_{B} = \epsilon\mathbb{Z}^n \cap \left(\Big(\textrm{Int} (\epsilon B C_n) \cup  
\bigcup_{i=1}^n \epsilon B \{(x_1\ldots x_n)\in  C_n; ~~ x_i=1\}\Big)
\setminus \epsilon B V_n\right),
\end{equation}
where $V_n$ is the set of vertices of the unit cube $C_n$.
For every $\tau\in V_B$, define $U^{\tau, B}_{\epsilon, |\xi_0|}\subset \Omega$ to be the union of all cells 
in $\tau + \epsilon B \mathbb{Z}^n$ that have nonempty intersection with $\Omega_{\epsilon\sqrt{n} |\xi_0|}$.
We extend the discrete deformation $u_\epsilon$ to the continuous function $u^{\tau, B}_{\epsilon, |\xi_0|}$  
on $U^{\tau, B}_{\epsilon, |\xi_0|}$, affine on all the simplices of the induced triangulation 
$\tau + B\mathcal{T}_{\epsilon, n}$.
We then have:
\begin{equation*}
\begin{split}
&\epsilon^{n}\sum_{j=0}^{n-1}\Big|\frac{|u_\epsilon(\tau + B(\alpha +
  \epsilon \sum_{i=1}^{j+1} e_{\pi(i)})) - u_\epsilon(\tau + B(\alpha
  + \epsilon \sum_{i=1}^{j} e_{\pi(i)}))|}{\epsilon| A(\tau + B(\alpha
  + \epsilon \sum_{i=1}^{j} e_{\pi(i)})) e_{\pi(j+1)}|} -1\Big|^2
\\ & \qquad\qquad \qquad\qquad
=  \frac{n!}{|\det B|}\int_{\tau + BT_\alpha^{\pi}}W(\nabla u^{\tau,
  B}_{\epsilon, |\xi_0|}(x)\lambda^{\tau, B}_{\epsilon,|\xi_0|}(x))~\mbox{d}x, 
\end{split}
\end{equation*}
where:
\begin{equation*}
\begin{split}
& \forall x\in (\tau + BT_\alpha^{\pi})\cap U^{\tau, B}_{\epsilon, |\xi_0|}\qquad \\
&\qquad \lambda^{\tau,B}_{\epsilon, |\xi_0|}(x) = |\xi_0| B\mbox{ diag} \left\{
\vert A(\tau + B(\alpha + \epsilon \sum_{i=1}^{\pi^{-1}(j)-1} e_{\pi(i)})) Be_j \vert^{-1}\right\}_{j=1}^n. 
\end{split}
\end{equation*}
Summing now over all simplices in the triangulations, we obtain the functional:
\begin{equation}\label{form_case}
\begin{split}
 I_{\epsilon, |\xi_0|} (u_\epsilon) = \sum_{\zeta\in
  S_{|\xi_0|}}\frac{1}{n! (nk)}\sum_{B\in K_\zeta} \frac{n!}{|\det B|} 
\sum_{\tau\in\{0\}\cup V_B}
\int_{U^{\tau, B}_{\epsilon, |\xi_0|}} W(\nabla u^{\tau, B}_{\epsilon, |\xi_0|}(x) \lambda^{\tau, B}_{\epsilon, |\xi_0|}(x)) ~ \mbox{d}x,
\end{split}
\end{equation}
and the bound:
\begin{equation}\label{estim_4}
\begin{split}
0 & \leq E_\epsilon(u_\epsilon)-I_{\epsilon, |\xi_0|}(u_\epsilon) \\ & \leq 
\sum_{\xi\in\mathbb{Z}^n, |\xi|=|\xi_0|} ~~\sum_{ \alpha \in
  R_\epsilon^\xi (\overline{\Omega \setminus \Omega_{\epsilon\sqrt{n} |\xi_0|}}) }
\epsilon^{n} \Big \vert \frac{ \vert u_\epsilon(\alpha + \epsilon \xi) 
- u_\epsilon (\alpha) \vert}{\epsilon \vert A(\alpha)\xi \vert} - 1 \Big \vert ^{2}.
\end{split}
\end{equation}
In (\ref{form_case}), $k$ is the number of non-zero entries in the vector $\zeta$, while
the factor $n!$ in the first denominator is due to the fact
that every edge in a given lattice is shared by $n!$ simplices in
$\mathcal{T}_{\epsilon, n}$.

\subsection{Case 5:  the general case of finite range interactions in
  $\mathbb{R}^n$} 

Reasoning as in the previously considered specific cases, we get:
\begin{equation}\label{error}
0 \leq E_\epsilon(u_\epsilon) - I_{\epsilon}(u_\epsilon)
\leq\sum_{\xi\in\mathbb{Z}^n, 1\leq |\xi| \leq M} ~~
\sum_{ \alpha \in R_\epsilon^{\xi} (\overline{\Omega \setminus
    \Omega_{\epsilon\sqrt{n}M}})} \epsilon^{n} \psi(|\xi|) \Big \vert \dfrac{ \vert
  u_\epsilon(\alpha + \epsilon \xi) - u_\epsilon (\alpha)
  \vert}{\epsilon \vert A(\alpha)\xi \vert} - 1 \Big \vert ^{2},
\end{equation}
where:
\begin{equation}\label{Iepsilon}
\begin{split}
I_{\epsilon}  = & \sum_{1\leq |\xi_0| \leq M} \psi(|\xi_0|) I_{\epsilon, |\xi_0|}.
\end{split}
\end{equation}

\section{Bounds on the variational limits of the lattice energies}\label{sec4}

Consider the following family of energies:
$$ F_\epsilon:L^2(\Omega,\mathbb{R}^n)\to\overline{\mathbb{R}}, \qquad 
F_\epsilon (u) = \left\{\begin{array} {ll}
    E_\epsilon(u_{|\epsilon\mathbb{Z}^n\cap\Omega}) & \mbox{ if }
    u\in\mathcal{C}(\Omega) \mbox{ is affine on }
    \mathcal{T}_{\epsilon, n}\cap\Omega\\
+\infty & \mbox{ otherwise}.\end{array}\right.. $$
By Theorem \ref{p}, the sequence $F_\epsilon$ has a subsequence (which we do not relabel)
$\Gamma$-converging to some lsc functional
$\mathcal{F}:L^{2}(\Omega,\mathbb{R}^n)\to \overline{\mathbb{R}}$. 
Our goal is to identify the limiting energy $\mathcal{F}$ in its exact
form, whenever possible, or find its lower and upper bounds. This will be accomplished in Theorem
\ref{thmain}, and in the next section.

\medskip

We first state some easy preliminary results regarding the quasiconvexification $QW$ and the piecewise affine extensions
$u^{\tau, B}_{\epsilon, |\xi_0|}$ of the discrete deformations $u_\epsilon$.

\begin{lemma}\label{quasiconv. W}
The quasiconvexification $QW: \mathbb{R}^{n \times n} \rightarrow \mathbb{R}$ of $W$ in (\ref{W}),
is a convex function, and:
\begin{equation}\label{QW}
QW(M) = \sum_{i=1..n; |Me_i|>1} (|Me_i|-1)^2 \qquad \forall
M\in\mathbb{R}^{n\times n}.
\end{equation}
\end{lemma}
\begin{proof}
By  Theorem 6.12 and Theorem 5.3 in  \cite{Dac} (see Theorem
\ref{nm1}) we note that:
$$QW(M) = \sum_{i=1}^n Q\big( |Me_i|-1\big)^2.$$
and that the convexification:
and the quasiconvexification $Qf$ of the
function $f:\mathbb{R}^n\rightarrow \mathbb{R}$ given by $f(\xi) =
(|\xi|-1)^2$ coincide with each other.
The claim follows by checking directly that:
$$Cf(\xi) = \left\{\begin{array}{ll} 0 & \mbox{ if } |\xi|\leq 1\\
(|\xi|-1)^2 & \mbox{ if } |\xi|>1. \end{array}\right. $$
\end{proof}

\begin{lemma}\label{auxiliary_functions}
For every $u \in W^{1, 2}(\Omega, \mathbb{R}^{n})$, and every
mesh-size sequence $\epsilon\to 0$, there exists a
subsequence $\epsilon$ (which we do not relabel) and a sequence $u_\epsilon \in
W_0^{1,2}(\mathbb{R}^n, \mathbb{R}^n)$ of continuous 
piecewise affine on the triangulation in $\mathcal{T}_{\epsilon, n}$ functions, such that:
\begin{equation*}
\begin{split}
\forall 1\leq |\xi_0| \leq M \quad \forall \zeta\in S_{|\xi_0|}  \quad \forall
B\in K_\zeta\quad &  \forall \tau\in \{0\}\cup V_{B} \\ & u = \lim_{\epsilon\to 0}
u_{\epsilon, |\xi_0|}^{\tau, B} ~~~ \mbox{ in } W^{1,2}(\Omega,\mathbb{R}^n).
\end{split}
\end{equation*}
\end{lemma}
\begin{proof}
Approximate $u$ by $u_k\in \mathcal{C}^\infty_0(\mathbb{R}^n,
\mathbb{R}^n)$, so that $u_k\to u$ in $W^{1,2}(\Omega,\mathbb{R}^n)$. Fix
$|\xi_0|\leq M$, $\zeta\in S_{|\xi_0|}, $ $B\in K_\zeta$ and $\tau\in
V_B$. Then, by the fundamental estimate of finite elements \cite{Ciar},
the $\mathbb{P}_1$-interpolation $u_{\epsilon, k}$ of $u_k$ on
$\mathcal{T}_{\epsilon, n}$, i.e. the continuous function affine on
the simplices in $\mathcal{T}_{\epsilon, n}$ which coincides with
$u_k$ on $\epsilon\mathbb{Z}^n$, satisfies:
$$\|u_{\epsilon, k} - u_k\|_{W^{1,2}(\Omega)}\leq \frac{1}{k} \qquad \forall
\epsilon\leq \epsilon_k.$$
Likewise, because the set of all involved quantities $|\xi_0|, \zeta, B, \tau$ 
is finite, it follows that:
$$ \|(u_{\epsilon, k})_{\epsilon, |\xi_0|}^{\tau, B} - u_k\|_{W^{1,2}(\Omega)}\leq
\frac{1}{k} $$ 
if only $\epsilon\leq \epsilon_k$ is sufficiently small. We set  $u_\epsilon:= u_{\epsilon_k, k}$ which 
satisfies the claim of the Lemma.
\end{proof}

We now observe a compactness property of $E_\epsilon$,
which together with the $\Gamma$-convergence of $F_\epsilon$ to
$\mathcal{F}$, implies convergence of the minimizers of $E_\epsilon$
to the minimizers of $\mathcal{F}$ (see Theorem \ref{thcom}).

\begin{lemma}\label{com}
Assume that $E_\epsilon(u_\epsilon)\leq C$, for some sequence of
discrete deformations $u_\epsilon:\epsilon
\mathbb{Z}^n\cap\Omega\to\mathbb{R}^n$, which we identify with
$u_\epsilon\in\mathcal{C}(\Omega)$ that are piecewise affine on
$\mathcal{T}_{\epsilon, n}\cap\Omega$ and agree with the discrete
$u_\epsilon$ at each node of the lattice.
Then there exist constants $c_\epsilon\in\mathbb{R}^n$ such that
$u_\epsilon- c_\epsilon$ converges (up to a subsequence) in
$L^{2}(\Omega, \mathbb{R}^n)$ to some $u\in W^{1,2}(\Omega, \mathbb{R}^n)$.
\end{lemma}
\begin{proof}
Observe that for every $|\xi_0|, \tau, B$ as in (\ref{Iepsilon}),
(\ref{form_case}), and every $\epsilon\leq \epsilon_0$:
\begin{equation}\label{m3}
\int_{U^{\tau, B}_{\epsilon, |\xi_0|}} W(\nabla u^{\tau,
  B}_{\epsilon, |\xi_0|}(x) \lambda^{\tau, B}_{\epsilon, |\xi_0|}(x))
~ \mbox{d}x \leq  C.
\end{equation}
Thus in particular, for some $\xi_0\in\mathbb{Z}^n$ such that
$\psi(|\xi_0|)\neq 0$, and for every $\eta>0$:
\begin{equation*}
\|\nabla u^{0,B}_{\epsilon, |\xi_0|}\|_{L^{2}(\Omega_\eta)} \leq C.
\end{equation*}
if only $\epsilon\leq \epsilon_0$ is small enough.
Fix $\eta>0$. The above bound implies that $\nabla
u^{0, B}_{\epsilon, |\xi_0|}$ converges weakly (up to a
subsequence) in $L^{2}(\Omega_\eta)$, which by means of the Poincar\'e
inequality yields weak convergence of $u^{0, B}_{\epsilon, |\xi_0|} -
c_\epsilon$ in $W^{1,2}(\Omega_\eta)$.
We now observe that:
$$\|u_{\epsilon, |\xi_0|}^{0, B} - u_\epsilon\|_{L^2(\Omega_\eta)}
\leq C \epsilon|\xi_0| \|u_\epsilon\|_{W^{1,2}(\Omega_\eta)},$$
because $u_{\epsilon, |\xi_0|}^{\tau, B}$ is a $\mathbb{P}_1$
interpolation of $u_\epsilon$ on the lattice
$\epsilon B\mathbb{Z}^n\cap \Omega_\eta$, allowing to use the
classical finite element error estimate in \cite[Theorem 3.1.6]{Ciar}. This ends the proof.
\end{proof}

\begin{theorem}\label{thmain}
We have:
\begin{equation}\label{bounds}
\forall u\in W^{1, 2}(\Omega, \mathbb{R}^{n}) \qquad I_Q(u) \leq
\mathcal{F}(u) \leq I(u),
\end{equation}
where:
\begin{equation}\label{Ilimit2}
\begin{split}
I_{Q}(u) & = \sum_{1\leq |\xi_0|\leq M} ~~\sum_{\zeta \in S_{|\xi_0|}, B\in K_\zeta} 
\psi(|\xi_0|) \frac{(1+ |V_{B}|)}{ (nk)
|\det B|}  \int_\Omega QW(\nabla u(x) \lambda^{B}_{|\xi_0|}(x))~\mathrm{d}x,\\
I(u) & = \sum_{1\leq |\xi_0|\leq M} ~~\sum_{\zeta\in S_{|\xi_0|}, B\in K_\zeta} 
 \psi(|\xi_0|) \frac{(1+ |V_{B}|)}{(nk) |\det B|} 
 \int_\Omega W(\nabla u(x) \lambda^{B}_{|\xi_0|}(x))~\mathrm{d}x,
\end{split}
\end{equation}
and where $\lambda^{B}_{|\xi_0|}(x)$ is given by:
\begin{equation}\label{lambda}
\lambda^{B}_{|\xi_0|}(x) = |\xi_0| B~\mathrm{diag}\left\{|A(x) Be_j|^{-1}\right\}_{j=1}^n.
\end{equation}
\end{theorem}
\begin{proof}
{\bf 1.} Let $u\in W^{1,2}(\Omega, \mathbb{R}^n)$ and consider the
approximating sequence $u_\epsilon$ as in Lemma
\ref{auxiliary_functions}. Directly from the definition of $\Gamma$-convergence
(see (\ref{def_gamma})), we obtain:
\begin{equation}\label{m1}
\mathcal{F}(u) \leq \liminf_{\epsilon\to 0} F_\epsilon(u_\epsilon) =
\liminf_{\epsilon\to 0} E_\epsilon(u_\epsilon).
\end{equation}
Further,  in view of the boundedness of $\psi$, and of the sequence
$\|\nabla u_\epsilon\|_{L^2(\Omega)}$, (\ref{error}) implies:
\begin{equation}\label{m2}
\begin{split}
0\leq E_\epsilon(u_\epsilon) - I_\epsilon(u_\epsilon) & \leq
C\sum_{\xi\in\mathbb{Z}^n, 1\leq |\xi| \leq M} ~~
\sum_{ \alpha \in R_\epsilon^{\xi} (\overline{\Omega \setminus
    \Omega_{\epsilon\sqrt{n}M}})} \epsilon^{n} \left( \Big \vert \dfrac{ 
  u_\epsilon(\alpha + \epsilon \xi) - u_\epsilon (\alpha)}{\epsilon
  |\xi|} \Big \vert^{2} + 1\right)\\ 
& \leq C\left(\|\nabla u_\epsilon\|^2_{L^2(\Omega\setminus\Omega_{\epsilon\sqrt{n}M})} +
  |\Omega\setminus\Omega_{\epsilon\sqrt{n}M}|\right) \to 0 \quad \mbox{ as }
\epsilon\to 0.
\end{split}
\end{equation}
Indeed, the third inequality in (\ref{m2}) can be proven by the same
argument as in the proof of Lemma
\ref{auxiliary_functions}. Alternatively, a direct proof can be
obtained as follows. Since $u_\epsilon$ is piecewise affine, we have:
$$\left|\frac{u_\epsilon(\alpha+\epsilon\xi) -
    u_\epsilon(\alpha)}{\epsilon|\xi|}\right|^2 =
\left|\int_0^1\langle \nabla u_\epsilon (\alpha+t\epsilon\xi),
  \frac{\xi}{|\xi|}\rangle~\mbox{d}t\right|^2 \leq \int_0^1 q_\epsilon
(\alpha+t\epsilon\xi)^2 ~\mbox{d}t,$$
where $q_\epsilon(p) = \sup_i \langle\nabla u_\epsilon(p), v_i\rangle$
when $p$ is an interior point of a face of the trangulation
$\mathcal{T}_{\epsilon, n}$ spanned by unit vectors $v_1,\ldots v_k$
(here $0\leq k\leq n$). Note that:
$$q_\epsilon(p)^2 \leq \frac{n!}{\epsilon^n}\int_T |\nabla
u_\epsilon|^2 \qquad \forall p\in T\in\mathcal{T}_{\epsilon, n}.$$
We hence obtain:
\begin{equation*}
\begin{split}
\forall 1\leq |\xi| \leq M \quad
&\sum_{ \alpha \in R_\epsilon^{\xi} (\overline{\Omega \setminus
    \Omega_{\epsilon\sqrt{n}M}})} \epsilon^{n} \left |\frac{u_\epsilon(\alpha + \epsilon \xi) - u_\epsilon (\alpha)}{\epsilon
  |\xi|} \right|^2 \leq \int_0^1 \sum_{ \alpha \in R_\epsilon^{\xi} (\overline{\Omega \setminus
    \Omega_{\epsilon\sqrt{n}M}})} \epsilon^n
q_\epsilon(\alpha+\epsilon\xi)^2~\mbox{d}t \\ & \quad \leq C\int_0^1 \left(
  \sum_{ \alpha} \int_T|\nabla
u_\epsilon|^2\right)~\mbox{d}t \leq C \int_0^1\|\nabla u_\epsilon\|^2_{L^2(\Omega \setminus 
    \Omega_{\epsilon\sqrt{n}M})} ~\mbox{d}t  =
\|\nabla u_\epsilon\|^2_{L^2(\Omega \setminus 
    \Omega_{\epsilon\sqrt{n}M})},
\end{split}
\end{equation*}
which achieves (\ref{m2}).

Consequently, by (\ref{m1}), (\ref{m2}), we see that:
$$\mathcal{F}(u) \leq\liminf_{\epsilon\to 0} I_\epsilon(u_\epsilon), $$
so that by Lemma \ref{auxiliary_functions} and
using the dominated convergence theorem, we obtain: 
\begin{equation}\label{sth}
\begin{split}
\mathcal{F}(u) \leq  \liminf_{\epsilon\to 0} \sum_{1\leq |\xi_0|\leq M}
~~\sum_{\zeta\in S_{|\xi_0|}, B\in K_\zeta} 
\psi(|\xi_0|) \frac{(1+ |V_{B}|)}{(nk) |\det B|} 
\int_\Omega W(\nabla u^{\tau, B}_{\epsilon, |\xi_0|}(x)
\lambda^{\tau,B}_{\epsilon, |\xi_0|}(x))~\mbox{d}x = I(u),
\end{split}
\end{equation}
in view of the uniform convergence of $\lambda_{\epsilon,
  |\xi_0|}^{\tau, B}$ to $\lambda_{|\xi_0|}^{B}$ in $\Omega$.
The proof of the upper bound for $\mathcal{F}$ in (\ref{bounds}) is hence accomplished.

\medskip

{\bf 2.} We now show the lower bound in (\ref{bounds}). Let $u \in W^{1,
  2}(\Omega, \mathbb{R}^{n})$; note that the upper bound proved above yields:
$\mathcal{F}(u)<\infty$. Therefore, $u$ has a recovery sequence $u_\epsilon
\in\mathcal{C}(\Omega)$ affine on $\mathcal{T}_{\epsilon,
  n}\cap\Omega$, such that: $u_\epsilon\to u$ in $L^2(\Omega,\mathbb{R}^n)$
and $E_\epsilon(u_\epsilon) \to\mathcal{F}(u)$ as $\epsilon\to 0$.

As in the proof of Lemma \ref{com}, we see that (\ref{m3}) holds for 
every $|\xi_0|, \tau, B$ as in (\ref{Iepsilon}),
(\ref{form_case}). Thus, for every $\eta>0$ we have:
\begin{equation}\label{bound}
\|\nabla u^{\tau, B}_{\epsilon, |\xi_0|}\|_{L^{2}(\Omega_\eta)} \leq C,
\end{equation}
for every $\epsilon\leq \epsilon_0$ is small enough.
Fix $\eta>0$. The bound (\ref{bound}) implies that every $\nabla
u^{\tau, B}_{\epsilon, |\xi_0|}$ converges weakly (up to a
subsequence) in $L^{2}(\Omega_\eta)$. Next, we note that $u^{\tau, B}_{\epsilon, |\xi_0|}$ 
converges to $u$ in $L^{2}(\Omega_\eta)$, which yields that the
same convergence is also valid weakly in $W^{1,2}(\Omega_\eta)$.

Indeed, by \cite[Theorem 3.1.6]{Ciar}, we have:
$$\|u_{\epsilon, |\xi_0|}^{\tau, B} - u_\epsilon\|_{L^2(\Omega_\eta)}
\leq C \epsilon|\xi_0| \|u_\epsilon\|_{W^{1,2}(\Omega_\eta)},$$
because $u_{\epsilon, |\xi_0|}^{\tau, B}$ is a $\mathbb{P}_1$
interpolation of $u_\epsilon$ on the lattice
$\epsilon B\mathbb{Z}^n\cap \Omega_\eta$. Consequently, in view of
(\ref{bound}):
$$\|u_{\epsilon, |\xi_0|}^{\tau, B} - u\|_{L^2(\Omega_\eta)} \leq 
\|u_{\epsilon, |\xi_0|}^{\tau, B} - u_\epsilon\|_{L^2(\Omega_\eta)} +
\|u_\epsilon- u \|_{W^{1,2}(\Omega_\eta)} \leq 
C \epsilon + \|u_\epsilon- u \|_{W^{1,2}(\Omega_\eta)}  \to 0,~ \mbox{
  as } \epsilon\to 0.$$

Since $QW\geq W$, we further obtain:
\begin{equation*}
\begin{split}
\mathcal{F}(u) & = \lim_{\epsilon \to 0}F_\epsilon(u_\epsilon) \geq
\limsup_{\epsilon \to 0}I_\epsilon({u_\epsilon}_{\mid \Omega_\eta})\\
& \geq \sum_{1\leq |\xi_0| \leq M} ~~ \sum_{\zeta\in
  S_{|\xi_0|}}\frac{\psi(|\xi_0|)}{(nk)} \sum_{B\in K_\zeta} \frac{1}{|\det B|} 
\sum_{\tau\in\{0\}\cup V_{l,B}}
\liminf_{\epsilon \rightarrow 0} \int_{\Omega_\eta}QW(\nabla u^{\tau, B}_{\epsilon, |\xi_0|}(x) \lambda^{\tau, B}_{\epsilon, |\xi_0|}(x))
~ \mbox{d}x \\  & \geq \sum_{1\leq |\xi_0| \leq M} ~~ \sum_{\zeta\in
  S_{|\xi_0|}, B \in K_\zeta} \frac{\psi(|\xi_0|)}{(nk)} \frac{1 + \vert V_{l, B} \vert}{|\det B|} 
\int_{\Omega_\eta} QW(\nabla u(x) \lambda^{B}_{|\xi_0|}(x)) ~
\mbox{d}x = I_Q(u_{\mid \Omega_\eta}),
\end{split}
\end{equation*}
where the last inequality above follows by the lower semicontinuity of the functional
$ \int_\Omega QW(v(x))~\mbox{d}x $
with respect to the weak topology of $L^{2}(\Omega_\eta, \mathbb{R}^{n
  \times n})$ (see Theorem \ref{convex}), and by the weak
convergence of $\nabla u^{\tau, B}_{\epsilon,
  |\xi_0|}\lambda^{\tau, B}_{\epsilon, |\xi_0|}$ to $\nabla u
\lambda^{B}_{|\xi_0|}$ in $L^{2}$. Since $\eta>0$ was
arbitrary, the proof is achieved.
\end{proof}

\begin{corollary}\label{Ffinite}
We have: $\mathcal{F}(u) < +\infty$ if and only if   $u \in W^{1, 2}(\Omega, \mathbb{R}^{n})$.
\end{corollary}
\begin{proof}
By Theorem \ref{thmain}, $\mathcal{F}$ is finite on all $W^{1,2}$ deformations.
Conversely, let $u\in L^2(\Omega,\mathbb{R}^n)$ and let
$\mathcal{F}(u)<\infty$. Then there exists a recovery sequence
$u_\epsilon\in\mathcal{C}(\Omega)$ affine on $\mathcal{T}_{\epsilon,
  n}\cap \Omega$, so that $u_\epsilon\to u$ in $L^2$ and
$F_\epsilon(u_\epsilon)$ is uniformly bounded. This implies (\ref{m3})
so in particular $\|\nabla u_\epsilon\|_{L^2(\Omega)}^2$ is bounded
and hence (up to a subsequence) $u_\epsilon$ converges weakly in
$W^{1,2}(\Omega)$. Consequently, $u\in W^{1,2}(\Omega)$.
\end{proof}

\begin{corollary}\label{lsc_envelope}
Let $G_0(I)$ denote the sequentially weak lsc envelope of $I$ in $W^{1,
  2}(\Omega, \mathbb{R}^{n})$. Then:
$$\mathcal{F}(u) \leq G_0(I)(u) \qquad \forall u \in W^{1, 2}(\Omega, \mathbb{R}^{n}).$$
\end{corollary}
\begin{proof}
The proof is immediate since the $\Gamma$-limit $F$ is sequentially weak lsc
in $W^{1, 2}(\Omega, \mathbb{R}^{n})$.
\end{proof}

\section{The case of nearest-neighbour interactions}\label{sec5}

In this section we improve the result in (\ref{bounds}) to the
exact form of the limiting energy $\mathcal{F}$, in the special cases of near
and next-to-nearest-neighbour interactions.

\begin{theorem}\label{Fnear} (Case 1: nearest-neighbour interactions in $\mathbb{R}^2$.)
Let $\Omega\subset\mathbb{R}^2$ and let $\psi(1)=1$ and
$\psi(|\xi|)=0$ for all $|\xi|\geq \sqrt{2}$. 
Denote: $\lambda (x) = \mathrm{ diag} \left\{
\vert A(x)e_1\vert^{-1},  \vert A(x)e_2\vert^{-1}\right\}.$ Then:
\begin{equation}\label{bz}
\mathcal{F}(u) = \left\{\begin{array} {ll} {\displaystyle{2\int_\Omega QW(\nabla u(x)
    \lambda(x))\mathrm{d}x}} & \mbox{ for } u \in W^{1, 2}(\Omega,
    \mathbb{R}^{2})\\
+\infty & \mbox{ for } u \in L^2\setminus W^{1,2}.\end{array}\right.
\end{equation}
\end{theorem}
\begin{proof}
From Theorem \ref{thmain} and (\ref{I1d}), we see that $I_Q(u) =
2\int_\Omega QW(\nabla u \lambda(x))~\mbox{d}x$ and $I(u) =
2\int_\Omega W(\nabla u \lambda(x))~\mbox{d}x$.
By Corollary \ref{lsc_envelope} it follows that:
$$\mathcal{F}(u) \leq G_0\Big(2\int_\Omega W(\nabla u
(x)\lambda(x))~\mbox{d}x\Big) = 2\int_\Omega QW(\nabla u \lambda(x))~\mbox{d}x. $$ 
The last equality is a consequence of Theorem \ref{qconvex} because
the function $f(x, M) = W(M \lambda(x))$ clearly satisfies the bounds
(\ref{mass}) and also its quasiconvexification with respect to $M$ equals:
$$Qf(x, M) = QW(M\lambda(x)).$$
The proof is now complete in view of Corollary \ref{Ffinite}.
\end{proof}

\begin{theorem}\label{th2} (Case 2: nearest-neighbour interactions in $\mathbb{R}^n$.)
Let $\Omega\subset\mathbb{R}^n$ and let $\psi(1)=1$ and $\psi(|\xi|)=0$
for $|\xi|\geq \sqrt{n}$. Denote: $\lambda (x) = \mathrm{ diag} \left\{ |A(x)e_j|^{-1}\right\}_{j = 1}^{n}$.
Then, the $\Gamma-$limit $\mathcal{F}$ has the form as in (\ref{bz}):
\begin{equation}\label{bz2}
\mathcal{F}(u) = \left\{\begin{array} {ll} {\displaystyle{2\int_\Omega QW(\nabla u(x)
    \lambda(x))~\mathrm{d}x}} & \mbox{ for } u \in W^{1, 2}(\Omega,
    \mathbb{R}^{n})\\
+\infty & \mbox{ for } u \in L^2\setminus W^{1,2}.\end{array}\right.
\end{equation}
\end{theorem}
\begin{proof}
The proof follows exactly as in Theorem \ref{Fnear}, using the
representation developed in section \ref{rep2}. Alternatively, using
the notation and 
setting of section \ref{rep3}, we see that $S_1=\{e_i\}_{i=1}^n$ and:
$$\forall \zeta\in S_1\quad N_\zeta = N=\{e_i, -e_i\}_{i=1}^n, ~~~\mbox{
  and }  K = \bigcup_{\zeta\in S_1} K_\zeta = \{B=\pm[e_i,e_{i+1},\ldots, e_{i-1}]\}_{i=1}^n,$$
so that $|K|=2n$. Also, for every $B\in K$ as above: $V_B=\emptyset$,
$|\det B| = 1$ and $\lambda_1^B(x) =
B\mbox{diag}\{|A(x)Be_j|^{-1}\}_{i=1}^n$, i.e. $\lambda_1^B(x)$ differs
from $\lambda(x)$ only by the order and sign of its columns. Hence:
$$\forall B\in K\qquad QW(\nabla u(x) \lambda_1^B(x)) = QW(\nabla u(x)
\lambda(x)), \quad W(\nabla u(x) \lambda_1^B(x)) = W(\nabla u(x) \lambda(x))$$
and so:
$$I_Q(u) = \sum_{\zeta\in S_1, B\in K_\zeta} \frac{1}{n}\int_\Omega QW(\nabla
u(x) \lambda_1^B(x))~\mbox{d}x = 2\int_\Omega QW(\nabla u(x)
    \lambda(x))~\mathrm{d}x.$$
Likewise: $I(u) = 2\int_\Omega W(\nabla u(x) \lambda(x))~\mathrm{d}x.$
The proof follows now by Corollary \ref{lsc_envelope} and Theorem
\ref{qconvex}, as before.
\end{proof}

Using the integral representation of section \ref{sec3}, we also
arrive at:

\begin{theorem}\label{th3} (Case 3: next-to-nearest-neighbour interactions in $\mathbb{R}^2$.)
Let $\Omega\subset\mathbb{R}^2$ and assume that $\psi(\sqrt{2})=1$ and
$\psi(|\xi|)=0$ for all $|\xi|\geq \sqrt{3}$ and $|\xi|\leq 1$. 
Denote:
\begin{equation*}
\lambda_{\sqrt{2}}(x) = \sqrt{2}B~\mathrm{ diag}
\left\{|A(x)Be_1|^{-1}, |A(x)Be_2|^{-1}\right\}, \qquad B = \left[\begin{array}{cc}
1 & -1 \\  1 & 1 \end{array}\right]. 
\end{equation*}
Then: 
\begin{equation*}
\mathcal{F}(u) = \left\{\begin{array} {ll} {\displaystyle{2\int_\Omega QW(\nabla u(x)
    \lambda_{\sqrt{2}}(x))\mathrm{d}x}} & \mbox{ for } u \in W^{1, 2}(\Omega,
    \mathbb{R}^{2})\\
+\infty & \mbox{ for } u \in L^2\setminus W^{1,2}.\end{array}\right.
\end{equation*}
\end{theorem}

The functionals $\mathcal{F}$ obtained in Theorems \ref{Fnear}, \ref{th2} and \ref{th3},
measure the deficit of a deformation $u$ from being an
orientation preserving (modulo compressive maps, due to the
quasiconvexification of the energy density $W$) realisation of the
metric $\bar G =
(\lambda^{-1})^{T}(\lambda^{-1})$. In the next section we compare
these functionals with the non-Euclidean energy $\mathcal{E}$. 

\section{Comparison of the variational limits and the energy $E$}\label{sec6}

In this section we assume that $\Omega$ is an open bounded subset of
$\mathbb{R}^2$. Our scope is to compare the following integral functionals: 
\begin{equation*}
\mathcal{F}_1(u) = \int_\Omega QW(\nabla u \lambda(x)) ~\mbox{d}x,\quad
\mathcal{F}_{\sqrt{2}}(u) = \int_\Omega QW(\nabla u  \lambda_{\sqrt{2}}(x))~\mbox{d}x,\quad
\mathcal{E}(u) = \int_\Omega \overline{W}(\nabla u A(x)^{-1})~\mbox{d}x,
\end{equation*}
where the stored energy density $\overline{W}: \mathbb{R}^{2 \times 2}
\rightarrow \overline{\mathbb{R}}_+$ satisfies (\ref{assu}).

\begin{lemma}\label{uno}
Assume that $\min \mathcal{E}(u) = 0$, so that the prestrain metric $G$ is
realisable by a smooth $u: \Omega \to \mathbb{R}^{2}$ with $(\nabla
u)^T\nabla u = G$. Then: $\mathcal{F}_1(u) = 0$.
\end{lemma}
\begin{proof}
Since $A=\sqrt{G}=\sqrt{(\nabla u)^T\nabla u}$, it follows that
$A=R\nabla u$, for some rotation field $R:\Omega\to SO(2)$.
Hence, $|A(x)e_i| = |\nabla u(x) e_i|$, and so both columns of the matrix:
$$\nabla u(x) \lambda(x) = 
\left[\frac{\nabla u(x) e_1}{|\nabla u(x) e_1|}, ~\frac{\nabla u(x) e_2}{|\nabla u(x) e_2|}\right]$$
have length $1$. The claim follows now by Lemma \ref{quasiconv. W}.
\end{proof}

The following example shows that $G$ may be realisable, as in Lemma
\ref{uno}, but the metric $\bar G={\lambda^{-1,T}}\lambda^{-1}$ is still
not realisable. The vanishing of the infimum of the derived
energy $\mathcal{F}_1$ is hence due to the quasiconvexification effect
in the energy density. 

\begin{example} 
Let $g:\mathbb{R}\to(0,+\infty)$ be a smooth function. Consider:
$$ G(x_1, x_2) = \left[\begin{array}{cc} 1/2 &  1 \\ 
1 & g(x_1)  \end{array}\right], \quad 
\bar{G}(x_1, x_2) = \mbox{diag}\{|A(x_1)e_1|^2, |A(x_1)e_2|^2\} = \left[\begin{array}{cc} 1/2 & 0  \\ 
0 & g(x_1)  \end{array}\right],$$
where the formula for $\bar{G}$ follows from the fact that $|A(x)e_i|^2 =
\langle e_i, A(x)^2 e_i\rangle = \langle e_i, G(x) e_i\rangle$.
We now want to assign $g$ so that the Gaussian cuvatures $\kappa$ and
$\kappa_1$ of $G$ and $\bar{G}$, satisfy:
\begin{equation}\label{K}
\kappa = 0, \qquad \kappa_1\neq 0.
\end{equation}
By a direct calculation, we see that:
\begin{equation*}
\begin{split}
& \kappa_1 = \frac{1}{\sqrt{g}} \left(\frac{g'}{\sqrt{g}}\right)' =
\frac{-2g g'' + (g')^2}{2g^2} \\ 
& (\frac{g}{2} - 1)^2 \kappa = -\frac{1}{2}g'' (\frac{g}{2} - 1) +
\frac{1}{8}(g')^2 = \frac{1}{2}g'' + \frac{g^2}{4}\kappa_1.
\end{split}
\end{equation*}
Hence, (\ref{K}) is equivalent to:
\begin{equation}\label{K2}
g>2, \qquad g''\neq 0, \qquad g''=\frac{(g')^2}{2(g-2)}.
\end{equation}
Clearly, the second order ODE above has a solution on a
sufficiently small interval
$(-\epsilon, \epsilon)$, for any assigned initial data $g(0) = g_0>2$ and
$g'(0) = g_1>0$. Also, this local
solution satisfies all three conditions in (\ref{K2}) by continuity,
if $\epsilon>0$ is small enough.

This completes the example. By rescaling $\tilde g(x_1) = g(\epsilon
x_1)$, we may obtain the metric
$G$ on $\Omega = (0,1)^2$, with the desired properties.
\end{example}

The next example shows that the induced metric $\bar G$ can be
realisable even when $G$ is not. In this case, one trivially has:
$\inf \mathcal{E}(u)>0$ while $\min \mathcal{F}_1(u) = 0$.

\begin{example} 
Let $w:(0,1)^2\to (0,\frac{\pi}{2})$ be a smooth function such that
$w_{x_1, x_2}\neq 0$, and define:
$$ G(x) = \left[\begin{array}{cc} 1 &  \cos w(x) \\ 
\cos w(x) & 1  \end{array}\right],\qquad \bar{G}(x) =
\mbox{diag}\{|A(x)e_1|^2, |A(x)e_2|^2\} = \mbox{Id}_2.$$
Clearly, $\kappa_1\neq 0$. We now compute the Gaussian curvature of $G$:
\begin{equation*}
\kappa = \frac{1}{\sin^4w} \left((-(\cos w) w_{x_1} w_{x_2} - (\sin w)
  w_{x_1, x_2})\sin^2w + (\sin^2w) w_{x_2} (\cos w) w_{x_1}\right) =
-\frac{w_{x_1, x_2}}{\sin w}\neq 0.
\end{equation*}
\end{example}

\medskip

The following simple observation establishes the relation between $\mathcal{F}_1$ and $\mathcal{F}_{\sqrt{2}}$.

\begin{lemma}
Let $\Omega=B(0,1)$. Then, we have:
$$ \forall u \in W^{1, 2}(\Omega, \mathbb{R}^{2})\qquad
\mathcal{F}_{\sqrt{2}}(u) = \overline{\mathcal{F}}_{1} (\sqrt{2}u \circ R), $$
where $\overline{\mathcal{F}}_1$ is defined with respect to the metric $G_1$ in:
$$G_1(x) = R^T G(Rx) R, \qquad R=\frac{1}{\sqrt{2}} B.$$
\end{lemma}
\begin{proof}
Note first that $G_1$ is the pull-back of the metric $G$ under the
rotation $x\mapsto Rx$. Thus:
\begin{equation*}
\begin{split}
\mathcal{F}_{\sqrt{2}}(u) &= \int_\Omega QW(\nabla u (x)
\lambda_{\sqrt{2}}(x))~\mbox{d}x \\ &
= \int_{\Omega}QW\Big(\sqrt{2} \nabla u (Ry) \sqrt{2} R~
\mbox{diag}\{|A(Ry)Be_1|^{-1},  |A(Ry)Be_2|^{-1}\}\Big) ~\mbox{d}y \\ & = 
\int_{\Omega}QW\Big(\nabla (\sqrt{2} u \circ R)(y) ~ \mbox{diag}\{|A(Ry)Re_1|^{-1},
|A(Ry)Re_2|^{-1}\}\Big) ~\mbox{d}y 
\\ & = \int_{\Omega}QW\Big(\nabla (\sqrt{2} u \circ R)(y)
\bar\lambda(y) \Big) ~\mbox{d}y = \overline{\mathcal{F}}_{1} (\sqrt{2}u \circ R),
\end{split}
\end{equation*}
because $|\sqrt{G_1(x)}e_i| = |A(Rx)Re_i|$, which implies:
$\bar \lambda(x) = \mbox{diag}\{|A(Rx)Re_1|^{-1},  |A(Rx)Re_2|^{-1}\}$.
\end{proof}

Finally, observe also that if $\mathcal{F}(u)=\mathcal{F}_1(u)=0$, then the length of columns in the
matrix $\nabla u(x) \lambda_{\sqrt{2}}(x)$ equals $\sqrt{2}$. Hence
$\mathcal{F}_{\sqrt{2}}(u)\neq 0$.

\section{Appendix}\label{secap}

\subsection{$\Gamma-$convergence} 
We now recall the definition and some basic
properties of $\Gamma$-convergence, that will be needed in the sequel.

\begin{definition} \label{def_gamma}
Let $\{I_\epsilon\}, I: X \rightarrow \overline{\mathbb{R}} = \mathbb{R} \cup \{-\infty, \infty\}$ 
be functionals on  a metric space $X$.
We say that $I_\epsilon$ $\Gamma$-converge to $I$ (as $\epsilon\to 0$), iff:

(i)  For every $\{u_\epsilon\}, u \in X$ with $u_\epsilon\rightarrow u$, we have:
$I(u) \leq \liminf_{\epsilon\to 0} I_\epsilon(u_\epsilon).$

(ii) For every $u \in X$, there exists a sequence $u_\epsilon\rightarrow u$ such that
$ I(u) = \lim_{\epsilon\to 0}I_\epsilon(u_\epsilon)$
\end{definition} 

\begin{theorem} \cite[Chapter 7]{Braides}\label{thcom}
Let $I_\epsilon, I$ be as in Definition \ref{def_gamma} and assume
that there exists a compact set $K \subset X$ satisfying:
$$ \inf_X I_\epsilon = \inf_KI_\epsilon \qquad \forall \epsilon.$$
Then: $ \lim_{\epsilon\to 0} (\inf_X I_\epsilon) = \min_X I$, and 
moreover if $\{u_\epsilon\}$ is a converging sequence such that: 
$$ \lim_{\epsilon\to 0} I_\epsilon(u_\epsilon) = \lim_{\epsilon\to 0} (\inf_X I_\epsilon), $$
then $u=\lim u_\epsilon$ is a minimum of $I$, i.e.: $I(u) = \min_X I$.
\end{theorem}

\begin{theorem}\label{p} \cite[Chapter 7]{Braides}
Let $\Omega$ be an open subset of $\mathbb{R}^n$. Any sequence of functionals
$I_\epsilon:L^2(\Omega,\mathbb{R}^n)\rightarrow \overline{\mathbb{R}}$ has a
subsequence which $\Gamma$-converges to some lower semicontinuous
functional $I:L^2(\Omega,\mathbb{R}^n)\rightarrow
\overline{\mathbb{R}}$.  Moreover, if every subsequence of $\{I_\epsilon\}$   
has a further subsequence that $\Gamma$-converges to (the same limit)
$I$, then the whole sequence $I_\epsilon$ $\Gamma$-converges to $I$.
\end{theorem}
 
\subsection{Convexity and quasiconvexity}

In this section $f:\mathbb{R}^{m\times n}\rightarrow
{\mathbb{R}}$ is a function assumed to be Borel measurable, locally bounded and bounded from below. 
Recall that the convex and quasiconvex envelopes of $f$, i.e. 
$Cf, Qf:\mathbb{R}^{m\times n}\rightarrow {\mathbb{R}}$ are defined by:
\begin{equation*}
\begin{split}
Cf(M) & =  \textrm{sup} \left\{g(M); ~~g:\mathbb{R}^{m \times n} \rightarrow 
\mathbb{R},~~ g \textrm{ convex,  } g \leq f\right\}, \\
Qf(M) & = \textrm{sup} \left\{g(M); ~~g:\mathbb{R}^{m \times n} \rightarrow 
\mathbb{R},~~ g \textrm{ quasiconvex,  } g \leq f\right\}. 
\end{split}
\end{equation*}
We say that $f$ is quasiconvex, if: 
\begin{eqnarray*}
f(M) \leq \fint_{D} f(M + \nabla \phi
(x))~\mbox{d}x \qquad \forall M \in \mathbb{R}^{m \times n}\quad \forall 
\phi \in W^{1, \infty}_0 (D, \mathbb{R}^{m}),
\end{eqnarray*} 
on every open bounded set $D\subset\mathbb{R}^{n}$.

\begin{theorem}\label{nm1}\cite[Chapter 6]{Dac}
\begin{itemize}
\item[(i)] When $m=1$ or $n=1$ then $f$ is quasiconvex if and only if $f$ is convex.
\item[(ii)]  For any open bounded $D\subset \mathbb{R}^n$ there holds:
$$Qf(M) = \inf\left\{\fint_D f(M+\nabla \phi(x))~\mathrm{d}x; ~ \phi\in W_0^{1,\infty}(D,\mathbb{R}^m)\right\}.$$
\item[(iii)] Assume that, for some $n_1+n_2 = n$ we have:
$$f(M) = f_1(M_{n_1}) + f_2(M_{n_2}) \qquad \forall M\in \mathbb{R}^{m\times n},$$
where $M_{n_1}$ stands for the principal minor of $M$ consisting of
its first $n_1$ columns, while $M_{n_2}$ is the minor of $M$
consisting of its $n_2$ last columns. Assume that $f_1, f_2$ are Borel
measurable and bounded from below. Then:
$$Cf= Cf_1 + Cf_2, \qquad Qf= Qf_1 + Qf_2$$
\end{itemize}
\end{theorem}


The following classical results explain the role of convexity and
quasiconvexity in the integrands of the typical integral functionals.

\begin{theorem}\label{convex}\cite{Dac}
Let  $\Omega$ be a bounded open set in $\mathbb{R}^{n}$ and let $f:
\mathbb{R}^{m\times 1} \rightarrow \mathbb{R}$ 
be lower semicontinuous (lsc).
Then the functional:
$$ I(u) = \int_{\Omega}f(u(x))~\mathrm{d}x \qquad \forall u \in L^{2}(\Omega, \mathbb{R}^{m})$$
is sequentially lsc with respect to the weak convergence
in $L^{2}(\Omega, \mathbb{R}^{m})$ if and only if $f$ is convex.
\end{theorem}

\begin{theorem}\label{qconvex}\cite[Chapter 9]{Dac}
Let  $\Omega$ be a bounded open set in $\mathbb{R}^{n}$ and let $f:
\Omega\times \mathbb{R}^{m\times n} \rightarrow \mathbb{R}$ be Caratheodory,
and satisfying the uniform growth condition:
\begin{equation}\label{mass}
\exists C_1, C_2>0 \quad \forall x\in\Omega \quad \forall M\in
\mathbb{R}^{m\times n}\qquad  C_1|M|^2 - C_2\leq f(x, M) \leq C_2(1 +
|M|^{2}).
\end{equation}
Assume that the quasiconvexification $Qf$ of $f$ with respect to the
variable $M$, is also a Caratheodory function.
Then for every $u\in W^{1,2}(\Omega, \mathbb{R}^m)$ there exists a
sequence $\{u_\epsilon\}\in u+ W_0^{1,2}(\Omega, \mathbb{R}^m)$ such
that, as $\epsilon\to 0$:
\begin{equation*} 
 u_\epsilon \rightharpoonup u \quad\mbox{weakly in } W^{1,2}
\quad \mbox{ and } \quad \int_{\Omega}f(x, \nabla u_\epsilon(x))~\mathrm{d}x \to \int_{\Omega}Qf(x, \nabla u(x))~\mathrm{d}x.
\end{equation*}
\end{theorem}

 \end{document}